\documentclass{article}

\usepackage[english]{babel}
\usepackage{amsthm}
\usepackage{amsmath}
\usepackage{amsfonts}
\usepackage{amssymb}
\usepackage{graphicx}
\usepackage[dvipsnames]{xcolor}
\usepackage{float}
\usepackage{indentfirst}
 \usepackage{theoremref}

\newtheorem{theorem}{Theorem}[section]

\newtheorem{lemma}[theorem]{Lemma}

\theoremstyle{definition}
\newtheorem{definition}{Definition}[section]

\theoremstyle{remark}

\title{Intrinsic Linking of 2-Complexes in $\mathbb{R}^4$}
\author{Nathan Huber, Ishaan Raghavendra Rao, \\Hannah Schwartz, and Tanishga Thankaraj Vijay}
\date{}

\begin{document}

\maketitle

\begin{abstract}
    We produce an infinite family of $2$-complexes that are intrinsically linked when embedded into four dimensions. In particular, we show that any embedding into  $\mathbb{R}^4$ of the suspension of a graph containing $K_6$ as a minor contains a non-trivially linked 1 and 2-cycle.  
\end{abstract}

\section{Introduction and Motivation}

A graph is intrinsically linked if, for any embedding of the graph into $\mathbb{R}^3$, there exists at least one pair of non-trivially linked cycles. The study of intrinsic linking stemmed from work of Conway and Gordon \cite{CandG} and Sachs \cite{sachs} in the 1980s showing the complete graph on 6 vertices, $K_6$, has this property. Research on spatial graphs remains active, see \cite{Adams97}, \cite{Adams3}, \cite{3Link}, \cite{Arbitrary}, and \cite{Complete} for example.   

Attempting to generalize intrinsic linking by simply embedding a graph into higher dimensions yields rather uninteresting results, as non-trivial $1$-dimensional links do not exist in $\mathbb{R}^n$ for $n>3$. Instead, it is natural to investigate how the notion of intrinsic linkedness extends to $n$-dimensional cell complexes (unions of $0$, $1 \dots n$-cells) embedded into higher dimensions. Most generally,  one can consider intrinsic linking of $k$-dimensional and $l$-dimensional subcomplexes in $\mathbb{R}^{k+l+1}$, see \cite{karasev2023converses} for example. 

In this paper, we consider intrinsically linked $2$-complexes in $\mathbb{R}^4$, meaning that each embedding of the complex contains a non-trivially  linked $1$ and $2$-cycle. In particular, we provide an infinite family of new examples of such $2$-complexes.
    
\begin{theorem} \label{mainthm}
    The suspension of any graph containing $K_6$ as a minor is an intrinsically linked $2$-complex in $\mathbb{R}^4$. 
\end{theorem}

It is known  that there exist intrinsically linked $n$-complexes in both $\mathbb{R}^{2n}$ and $\mathbb{R}^{2n+1}$ for any integer $n>0$. Examples covering all odd ambient dimension include those of Lov{\'a}sz and Schrijver \cite{lovasz}, Taniyama \cite{tani}, 
Skopenkov \cite{skope}, Tuffley \cite{Tuff} and most recently Nikkuni \cite{nikkuni}. A largely disjoint collection of work provides examples in the case of even ambient dimension, and addresses a related issue: whereas every $n$-complex embeds in $\mathbb{R}^{2n+1}$, not every $n$-complex embeds into $\mathbb{R}^{2n}$. Non-embeddability into $\mathbb{R}^{2n}$ occurs if and only if the complex has a non-zero Van Kampen obstruction  \cite{Kampen}, when $n \not = 2$. This fact was reproved by Freedman, Krushkal, and Teichner in \cite{freedman1994van} and shown \emph{not} to hold in the case $n=2$. Their construction involves the first example of an intrinsically linked $2$-complex in $\mathbb{R}^4$. 

In subsequent work, Freedman and Krushkal \cite{FreedmanKrushkal2014} were first to produce examples of intrinsically linked $n$-complexes in all even dimensions, and moreover, examples of $2$-complexes such that each embedding into $\mathbb{R}^4$ contains linked cycles with non-vanishing higher order linking numbers. These examples were recently studied further in \cite{komendarczyk2025milnorinvariantsthicknessspherical}. 
In 2025, Nikkuni \cite{nikkuni2} constructed \emph{minimal} intrinsically linked simplicial $n$-complexes in each 
$\mathbb{R}^{2n}$. Our examples are distinct from those originally presented by Nikkuni; in fact, shortly after the first version of our paper appeared, Nikkuni \cite{nikkuni2026intrinsiclinkingsimplicialncomplexes} not only generalized our main example to higher dimensions, but also proved that this example are minimal. 

 Our proof of an intrinsically linked $2$-complex involves a mod $4$ invariant based on the mod $2$ invariant originally used by Conway and Gordon to prove the intrinsic linkedness of $K_6$ in \cite{CandG}. Rather than totaling the linking numbers over each pair of linked $1$-cycles of an embedded graph, we tally the linking numbers over each pair of linked $1$ and $2$-cycles in an embedded $2$-complex. 

In Section 2, we cover necessary background on classical results involving intrinsically linked spatial graphs. In Section 3, we introduce intrinsically linked 2-complexes in $\mathbb{R}^4$ and present our main construction and proof. A short section of possible future directions is also included for the interested reader. 
\\

\noindent
{\bf Acknowledgments.} 
All four authors would like to thank the North Carolina School of Science and Mathematics, where they conducted this research, as well as our Dean of Mathematics Beth Bumgardner for supporting higher level research in public high school. They would also like to thank Slava Krushkal, Ryo Nikkuni, and Fedor Manin for helpful feedback on the initial draft. During this project, the author HS was supported by NSF grant DMS-1502525. 

\section{Necessary Background}

Our proof of the existence of intrinsically linked $2$-complexes in $\mathbb{R}^4$ will mimic much of Conway and Gordon's original strategy from \cite{CandG}, shifting their argument up by one dimension.

\begin{definition}
    A graph is \emph{intrinsically linked} in $\mathbb{R}^3$ if, for every possible embedding into $\mathbb{R}^3$, there exist two disjoint cycles within the graph that form a non-trivial link. 
\end{definition}

\begin{figure}\label{k6pic}
    \centering
\includegraphics[width=0.8\linewidth]{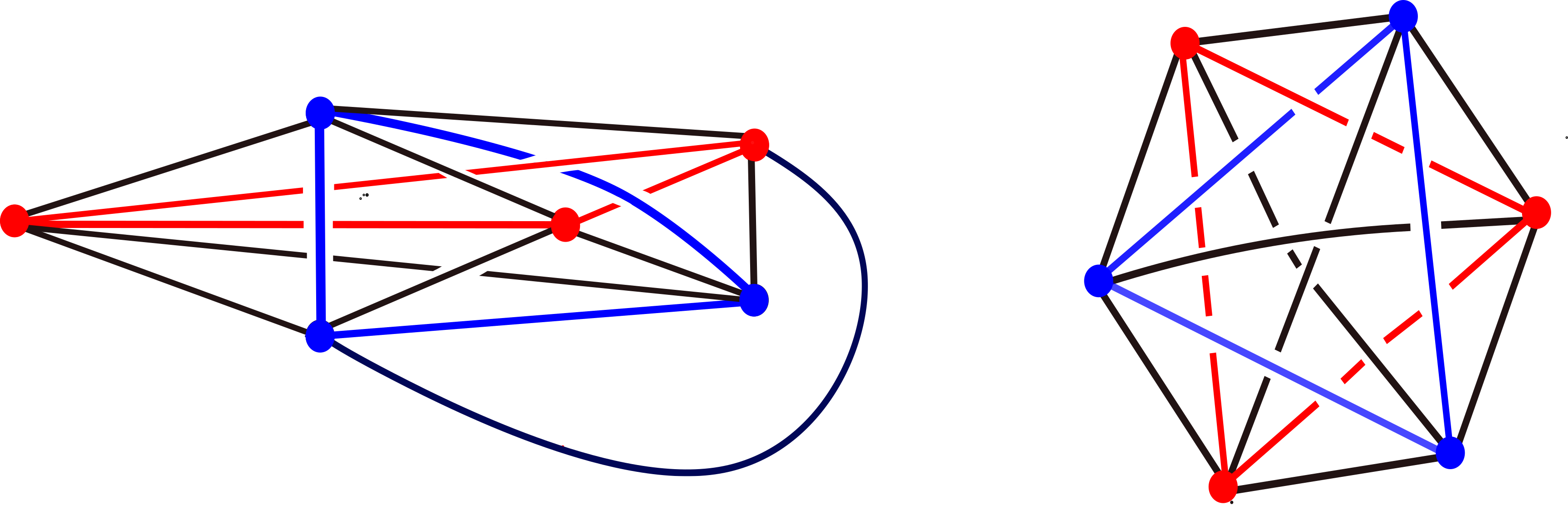}
    \caption{Two distinct spatial embeddings in $\mathbb{R}^3$ of the complete graph $K_6$. One pair of non-trivially linked $1$-cycles is highlighted in red and blue for each embedding.}
    \label{K6pic}
\end{figure}

Conway and Gordon \cite{CandG}, and independently Sachs \cite{sachs}, proved the following foundational result:

\begin{theorem}[Conway and Gordon, 1983; Sachs, 1984]
\label{intrinsically_linked_graphs}
Every spatial embedding of the complete graph $K_6$ contains a non-trivial link.
\end{theorem}

One example of an embedding of $K_6$ is illustrated in Figure \ref{K6pic}, with a pair of non-trivially linked cycles highlighted. Indeed, no matter the embedding, at least one pair of length 3 cycles will be non-trivially linked. Pairs of triangles are in fact the only kind of disjoint cycles in $K_6$. To see this, note that each component of the link must have length at \emph{least} 3 (this is the minimum length of a cycle in $K_6$). However, if either has length \emph{greater} than 3, it precludes any disjoint cycle from having this minimum length.

\begin{definition}
We call each length 3 cycle $T$ of $K_6$ a \emph{triangle}, and the unique disjoint length 3 cycle (the union of the edges between the three vertices not used in the first triangle) its \emph{dual triangle} $\bar T$. 
\end{definition}

There are exactly 10 distinct pairs of dual triangles in $K_6$. For a fixed embedding of $K_6$, Conway and Gordon extract an invariant $\lambda$ counting the linking number over all such pairs of triangles. 

\begin{definition}\label{cgdef}
    Let $h : K_6 \to \mathbb{R}^3$ be any embedding of $K_6$. Define
$$\lambda \equiv \sum_{i=1}^{10}{\omega_i }\pmod{2}$$
where $\omega_i$ is the linking number in $\mathbb{R}^3$ of the embedded links $h(T \sqcup \bar{T})$ for each of the 10 pairs of dual triangles $T$ and $\bar T$.
\end{definition} 

Note that in order to be well-defined, the tally must be taken modulo 2 as the graph itself is not equipped with an orientation. Conway and Gordon's proof from \cite{CandG} that any embedding of $K_6$ contains a non-trivial link consists of two main steps:

\begin{enumerate}
    \item Show that $\lambda(h)$ is non-vanishing for a fixed embedding $h$ of $K_6$, and that the intrinsic linkedness of the graph follows directly from this fact.
    \item Show that $\lambda$ does not depend on the embedding  $h$, but is in fact an invariant of the abstract graph $K_6$.
\end{enumerate}

It also follows that any graph containing $K_6$ as a minor is also intrinsically linked. Our proof of the existence of an intrinsically linked $2$-complex in the next section follows largely the same outline.

\section{Intrinsically Linked 2-Complexes in $\mathbb{R}^4$}

We first generalize the notion of intrinsic linking to one appropriate for ambient dimension $4$, where the linking dimensions are $1$ and $2$. 

\begin{definition} \label{IL2cx}
    We call a $2$-complex $\mathcal{C}$ \emph{intrinsically linked} in $\mathbb{R}^4$ if, for every possible embedding into $\mathbb{R}^4$, there exists a non-trivial link $C_1 \sqcup C_2 \subset \mathcal{C}$ within the complex, where $C_1$ is a $1$-cycle homeomorphic to a circle, and $C_2$ is a $2$-cycle homeomorphic to a closed, orientable surface. 
\end{definition}

Note that for technical reasons appearing in our main argument, as well as to streamline the setting, we restrict our attention to \emph{piecewise differentiable} embeddings of $2$-complexes. All ambient isotopies we consider between such complexes will be smooth. In contrast to dimension $3$, care must be taken to consider the differences between the smooth and topological categories. 

Our definition above differs slightly from its analogs in the previous related sources \cite{freedman1994van}, \cite{FreedmanKrushkal2014}, and \cite{nikkuni2} mentioned in the introduction: we choose to allow the $2$-cycle $C_2$ in the definition above to be a closed, orientable surface -- rather than specifically a $2$-sphere -- only because our proof works in either case. A basic understanding of computing the linking number between oriented curves and surfaces in $\mathbb{R}^4$ will be critical for our proofs that follow. Below, we shall detail the meaning of the linking number of a two component link of the form $C_1 \cup C_2 \subset \mathbb{R}^4$ as in Definition \ref{IL2cx}. Similar to our notation from the previous section, we continue to denote this integer by $\omega(C_1, C_2)$. 

As in the classical 3D case, this isotopy invariant of a link can be thought of homologically. Since $C_2$ has codimension two in $\mathbb{R}^4$, by Alexander duality  $H_1(\mathbb{R}^4-C_2) \simeq \mathbb{Z}$. An isomorphism is determined by choosing an orientation on the surface $C_2$ -- this specifies a ``positively oriented" meridian to $C_2$ that generates $H_1(\mathbb{R}^4-C_2)$, see Figure \ref{crossingchangehomotopy} for example. The linking number $\Omega(C_1, C_2)$ is the integer representing the class $[C_1] \in H_1(\mathbb{R}^4-C_2)$. The sign of the linking number of course also necessitates choosing an orientation on the curve $C_1$. Stated geometrically, the linking number can be thought of as the number of times (up to homotopy away from $C_2$) that the oriented curve $C_1$ wraps around the positive meridian of the surface $C_2$.

To find an intrinsically linked $2$-complex as in Definition \ref{IL2cx}, we begin by considering the suspension $\mathcal S(K_6)$ of the complete graph $K_6$, thought of as a $2$-complex with $1$-skeleton equal to the join of $K_6$ with two ``distinguished" vertices $a$ and $b$. The faces of the $2$-complex arise in pairs that correspond to each edge $e$ of $K_6$: the first face in each pair cones the edge $e$ to the vertex $a$, while the second face cones $e$ to the vertex $b$. An embedded version of this suspension is illustrated in Figure \ref{sigmadiagram}.

The close relationship between $K_6$ and its suspension allows us to utilize many properties of the former in our proof of the intrinsic linkedness of the latter. For instance, we can characterize all disjoint $1$ and $2$-cycles in the suspension, just as we characterized all disjoint pairs of $1$-cycles in $K_6$. 

\begin{lemma} \label{trianglelemma}
A $1$-cycle and $2$-cycle in the suspension $S(K_6)$ are disjoint if and only if the $1$-cycle is equal to a triangle $T$ in $K_6$, and the $2$-cycle is the suspension $S(\bar{T})$ of the dual triangle $\bar{T}$.
\end{lemma}

\begin{proof}
Let $C_1$ be a $1$-cycle of $S(K_6)$ disjoint from a $2$-cycle $C_2$. If $C_1 \not \subset K_6$, then this cycle must contain either $a$ or $b$, as those are the only two other points. However, all $2$-cycles of the suspension also contain both $a$ and $b$, contradicting the fact that $C_1$ and $C_2$ are disjoint. Therefore, $C_1 \subset K_6$. Moreover, $C_1$ must have length $3$: this is certainly the minimum length for any $1$-cycle, but also the maximum for this specific one. This holds since the $2$-cycle $C_2$ intersects $K_6$ in a $1$-cycle disjoint from $C_1$, and this cycle cannot coincide with those of $C_1$. 
\end{proof}

 This means that each of the 10 links of the form $T \cup \bar T$ in the embedding of $K_6$ corresponds to a \emph{pair} of linked cycles in $S(K_6)$, both $T \cup \mathcal{S}(\bar T)$ and $\bar T \cup \mathcal{S}(T)$. 
 
\begin{definition} \label{4dinvt} Given any embedding of $\mathcal S(K_6)$ into $\mathbb{R}^4$, define

$$\Lambda 
\equiv_2 \sum_{i=1}^{10}{\Omega_i}\pmod{2}$$
where the index $i$ ranges over all distinct pairs of dual triangles $T \cup \bar T$ in $K_6$, and each term $\Omega_i$ is equal to the sum $\frac{1}{2}(|\omega(T,\mathcal{S}(\bar T))| + |\omega(\bar T, \mathcal{S}(T))|)$. 
\end{definition}

This new invariant is similar to $\lambda$ from Definition \ref{cgdef}, but some key differences will be relevant. First, note that in this case, taking the absolute value of each linking number $\omega$ guarantees that $\Lambda$ is well-defined regardless of the arbitrary orientation given to each link in the $2$-complex. Meanwhile, the sum is considered modulo $2$ with a scaling factor of $\frac{1}{2}$ in order to obtain a nontrivial invariant regardless of the embedding. 

We shall work with a specific embedding $\Sigma_6: S(K_6) \hookrightarrow
\mathbb{R}^4$ constructed as follows. First, embed the copy of $K_6$ in the $1$-skeleton of $S(K_6)$ into the level set $\mathbb{R}^3 \times \{0\} \subset \mathbb{R}^4$. This can be done in many ways, and our proof is independent of this choice. Next, place the vertex $a$ into $\mathbb{R}^3 \times \{1\} \subset \mathbb{R}^4$ and $b$ into $\mathbb{R}^3 \times \{-1\} \subset \mathbb{R}^4$. Each face of the $2$-complex is embedded by coning an edge of $K_6$ either up to $a$ or down to $b$, as shown in Figure \ref{sigmadiagram}. Thus, at every level $\mathbb{R}^3 \times \{w\}$ for $w\in (-1, 1)$ there exists an embedded copy of $K_6$, until reaching levels $\pm 1$ where the vertices $a$ and $b$ are placed.  

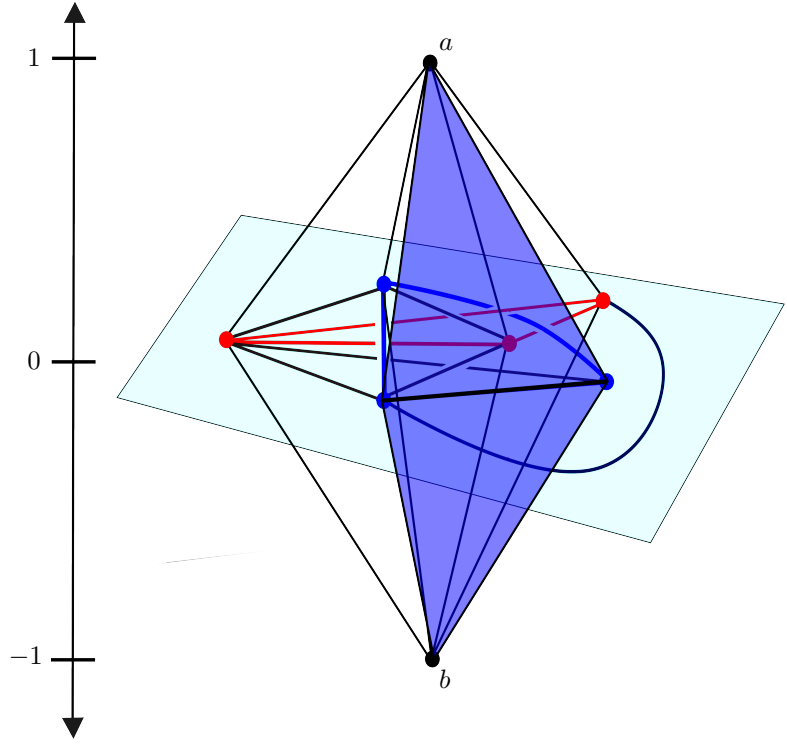
\begin{figure}[h]
    \centering
\includegraphics[width=0.8\linewidth]{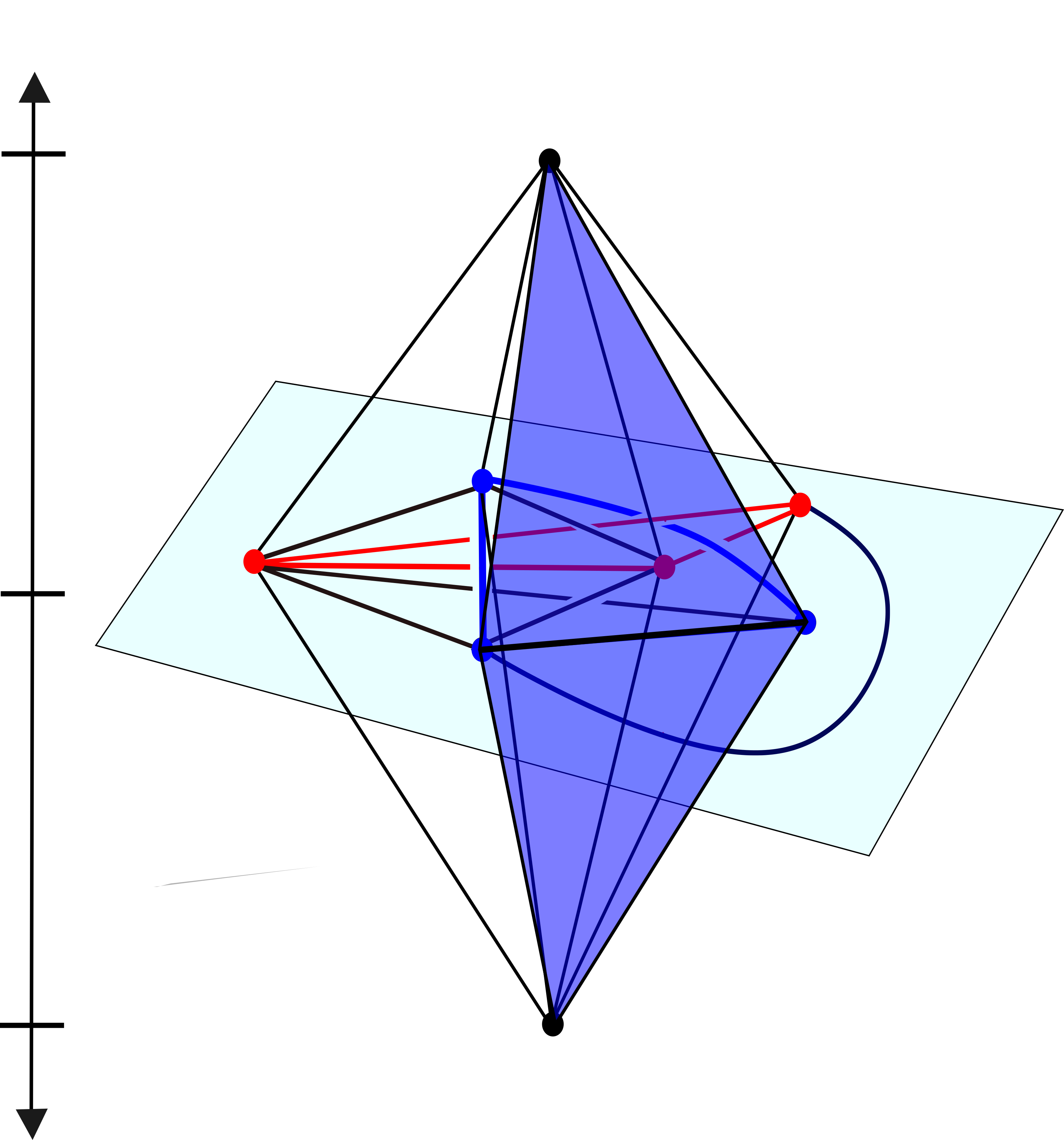}
\put(-285,139){$0$}
\put(-292,28){$-1$}
\put(-285,253){$1$}
\put(-130,259){$a$}
\put(-130,19){$b$}
    \caption{The embedding $\Sigma_6$ of the suspension $S(K_6)$ into $\mathbb{R}^4$. The slice of the embedding $\Sigma_6$ contained in  $\mathbb{R}^3 \times \{0\}$ is the sub-complex $K_6$, a subset of the $1$-skeleton of our cell structure on $S(K_6)$. The faces of $S(K_6)$ are embedded by coning each edge of $K_6$ up and down to the points $a$ and $b$, pictured at heights $1$ and $-1$ respectively. Examples of two such faces are shaded in the figure.}
    \label{sigmadiagram}
\end{figure}

\begin{lemma}\label{lem1}
The value of $\Lambda$ is non-trivial for $\Sigma_6$. Hence at least one pair of cycles in $\Sigma_6$ has non-zero linking number in $\mathbb{R}^4$.
\end{lemma}

\proof 
For this particular embedding, if $T \cup \bar T \subset K_6$ is a link with non-zero linking number in $\mathbb{R}^3$, then both $T \cup \mathcal{S}(\bar T)$ and $\bar T \cup \mathcal{S}(T)$ are non-trivially linked in $\mathbb{R}^4$. Indeed, the absolute values of the linking numbers $\omega(T,\mathcal{S}(\bar T))$ and $\omega(\bar T, \mathcal{S}(T))$ in $\mathbb{R}^4$ are both equal to the absolute value of the linking number $\omega(T,\bar T)$ of the link $T \cup \bar T$ in $\mathbb{R}^3$. This follows from the discussion on linking number presented below Definition \ref{IL2cx}, noting that an oriented meridian of $T$ or $\bar T$ in $\mathbb{R}^3 \times \{0\}$ is also a meridian of the suspension $\mathcal S(T)$ or $\mathcal S(\bar T)$, respectively.

 Therefore, $\Lambda \equiv_2\lambda$, since $\Omega_i \equiv_2 \omega_i$ for each $i\in \{1, \dots, 10\}$. Thus, $\Lambda \equiv_2 1$, as we already know that $\lambda\equiv_21$ by Conway and Gordon \cite{CandG}. \qed 

\begin{theorem} \label{mainproof}
     The value of $\Lambda$ does not depend on the (piecewise) smooth embedding of $\mathcal S(K_6)$ into $\mathbb{R}^4$. 
\end{theorem}

\proof Take an embedding $\Sigma'_6$ not isotopic to our standard one $\Sigma_6$. These embeddings are related by a homotopy from $\Sigma_6'$ to $\Sigma_6$. First, since homotopy implies isotopy for curves $\mathbb{R}^4$, the $1$-skeleton of $\Sigma'_6$ can be isotoped to the $1$-skeleton of $\Sigma_6$. This isotopy extends to one taking $\Sigma'_6$ to a new embedding $\Sigma''_6$ of $\mathcal S(K_6)$ whose faces remain distinct from those of $\Sigma_6$. 

\begin{figure}[h]
    \centering
\includegraphics[width=0.8\linewidth]{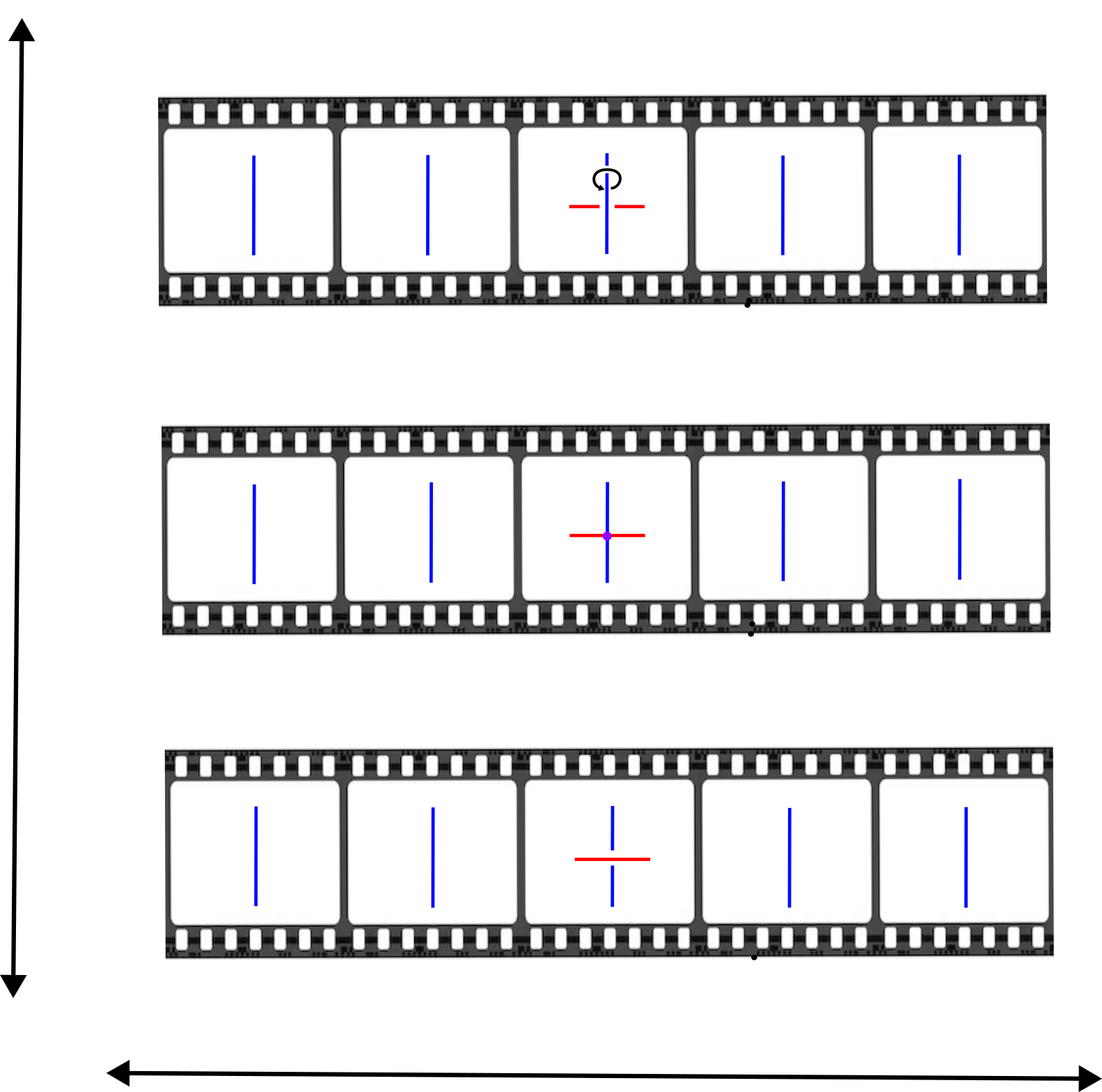}
\put(-285,140){$t$}
\put(-130,-10){$w$}
\put(-118,230){$m$}
    \caption{The local picture of a homotopy with one transverse intersection point between a $1$-cycle (shown in red) and a $2$-cycle (shown in blue). Here, $w$ denotes the ``extra" spatial dimension, while $t$ denotes time during the homotopy. Note that the $1$-cycles at the beginning and end of the homotopy differ by exactly one meridian $m$ of the $2$-cycle -- therefore the linking number between the $1$ and $2$-cycle after such a homotopy changes by $\pm 1$, depending on the orientation of the $1$-cycle and the meridian.}
\label{crossingchangehomotopy}
\end{figure}

However, each face of $\Sigma''_6$ is smoothly isotopic rel boundary to the corresponding face in $\Sigma_6$ bounded by the same cycle\footnote{Classically, Marumoto \cite{Maru} proved this fact in the topological category.}, by Theorem 10.9 of Gabai \cite{Gabai}. Performing these isotopies simultaneously gives a homotopy from $\Sigma''_6$ to $\Sigma_6$ fixing (but not supported away from) the $1$-skeleton. During the homotopy, each face $F \subset \mathcal{S}(K_6)$ may intersect another face or the interior of an edge $E$. After a perturbation of the homotopy, we may assume by general position that $F$ avoids $0$-cells completely during the homotopy. Moreover, we can arrange for the trace of the homotopy of $F$ in $\mathbb{R}^4$ to intersect the interior of each edge $E$ transversally in finitely many points during the homotopy. The local model of such a homotopy is shown in Figure \ref{crossingchangehomotopy}. 

Readers familiar with Conway and Gordon's argument in \cite{CandG} will begin to note the similarity of our proof with theirs. We must show that the value $\Lambda$ is unchanged throughout this homotopy. However, in contrast to their setting, our $2$-complex may be immersed at all times \emph{throughout} the homotopy, rather than just at finitely many times as in their case. Therefore, although $\Lambda$ has been defined only for embedded $2$-complexes in $\mathbb{R}^4$, for the purposes of our argument we extend its definition to $2$-complexes immersed in $\mathbb{R}^4$ whose only double points are between faces intersecting in their interiors. In this special case, the linking number between disjoint $1$ and $2$-cycles is still well-defined. 

This tally of linking numbers can only be changed at the finitely many times when a face $F$ passes through an edge $E$ completely disjoint from $F$ in $\mathcal{S}(K_6)$, i.e. an edge whose endpoints do not intersect the boundary of the face $F$. Consider such an edge $E$, illustrated in Figure \ref{labelledpieces}. The homotopy of $F$ can be modified near the intersection point with $E$ as shown in Figure \ref{deformation} by first pushing $F$ closer to one endpoint $p$ of $E$, and then replacing the intersection between $E$ and $F$ by four new intersection points, between $F$ and the four other edges of $\mathcal{S}(K_6)$ also incident with the vertex $p$. 

\begin{figure}[h]
    \centering
\includegraphics[width=0.4\linewidth]{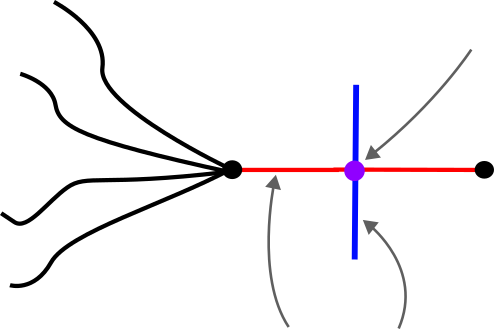}
\put(-140,91){$E_1$}
\put(-150,70){$E_2$}
\put(-60,-11){$E$}
\put(-77,34){$p$}
\put(-31,-11){$F$}
\put(-4,85){The intersection point $E \pitchfork F$}
\put(20,73){during the homotopy}
    \caption{Labeling components of the local diagram for the homotopies from the proof of Theorem \ref{mainproof} depicted in Figure \ref{deformation}.}
    \label{labelledpieces}
\end{figure}

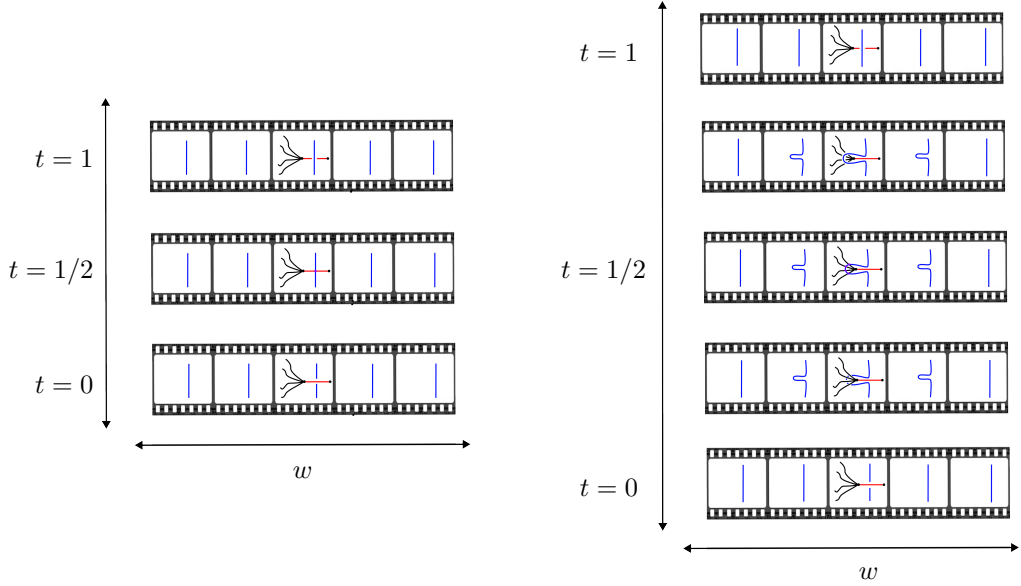
\begin{figure}[h]
    \centering
\includegraphics[width=1\linewidth]{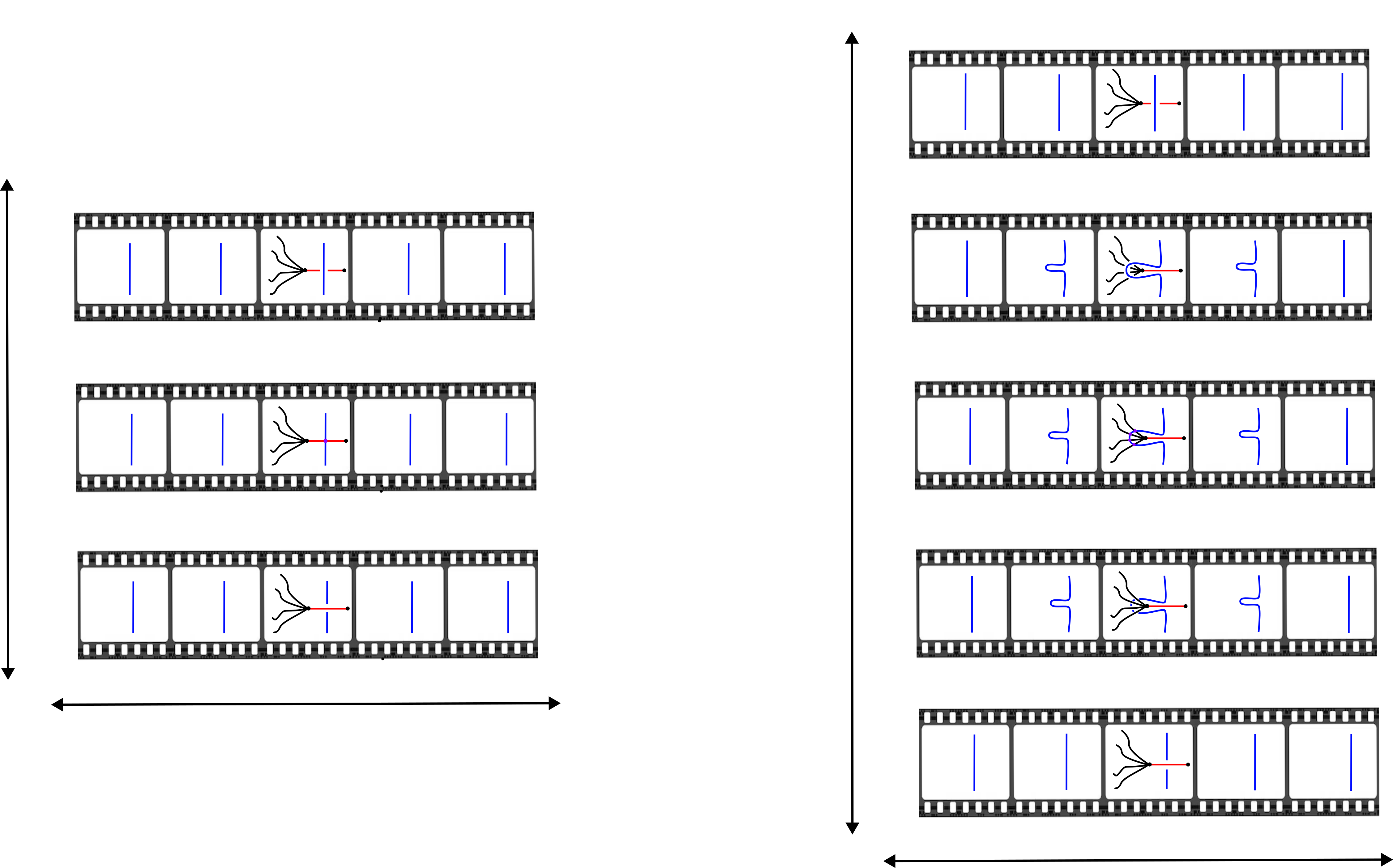}
\put(-273,28){$w$}
\put(-60,-10){$w$}
\put(-370,145){$t=1$}
\put(-380,103){$t=1/2$}
\put(-370,60){$t=0$}
\put(-165,185){$t=1$}
\put(-172,103){$t=1/2$}
\put(-165,21){$t=0$}
    \caption{The homotopies of the face $F$ both before (left) and after (right) the deformation from the proof of Theorem \ref{mainproof}. As in Figure \ref{crossingchangehomotopy}, the variable $w$ denotes the extra spatial dimension in $\mathbb{R}^4$, while $t$ gives the time during the homotopy. Note that this deformation is supported entirely for times $t \in [0,1]$, i.e. at times $t=0$ and $t=1$ of both homotopies the position of the $2$-complex is identical. Therefore, this local change does not affect the rest of the homotopy and can be made whenever an intersection point as in Figure \ref{labelledpieces} arises.}
    \label{deformation}
\end{figure}

Two of these four edges are not disjoint from the boundary of $F$, so their intersections with $F$ during the homotopy do not affect the linking numbers tallied in $\Lambda$. However, the remaining pair of edges, which we call $E_1$ and $E_2,$ are disjoint from $F$ (see Figure \ref{labelledpieces}). For both $i=1,2$, there are exactly two triangles $T$ and $T'$ in $K_6$ containing $E_i$ with dual triangles $\bar T$ and $\bar T'$ containing an edge of the face $F$, as shown in Figure \ref{disjointcycles}. 
By Lemma \ref{trianglelemma}, it follows that $T \cup \mathcal S(\bar T)$ and $T' \cup \mathcal S(\bar T')$ are the only two links containing both $E_i$ and $F$ in the suspension $\mathcal{S}(K_6)$. This pair of links are depicted in Figure \ref{disjointcycles}. 

The absolute values of the linking numbers $\omega(T, \mathcal S(\bar T))$ and $\omega(T', \mathcal S(\bar T'))$ will each change by $\pm 1$ after each intersection of $E_i$ with $F$ during the homotopy. Hence each intersection point of $F$ with $E_i$ during the homotopy contributes either $0$ or $1$ mod 2 to the sum $\Lambda$ from Definition \ref{4dinvt} (due to the factor of $\frac{1}{2}$ in each term $\Omega_i$). Therefore, the \emph{pair} of intersection points with $E_1$ and $E_2$ modifies $\Lambda$ by a total of $0$ mod $2$. 

\begin{figure}[h]
    \centering
\includegraphics[width=0.9\linewidth]{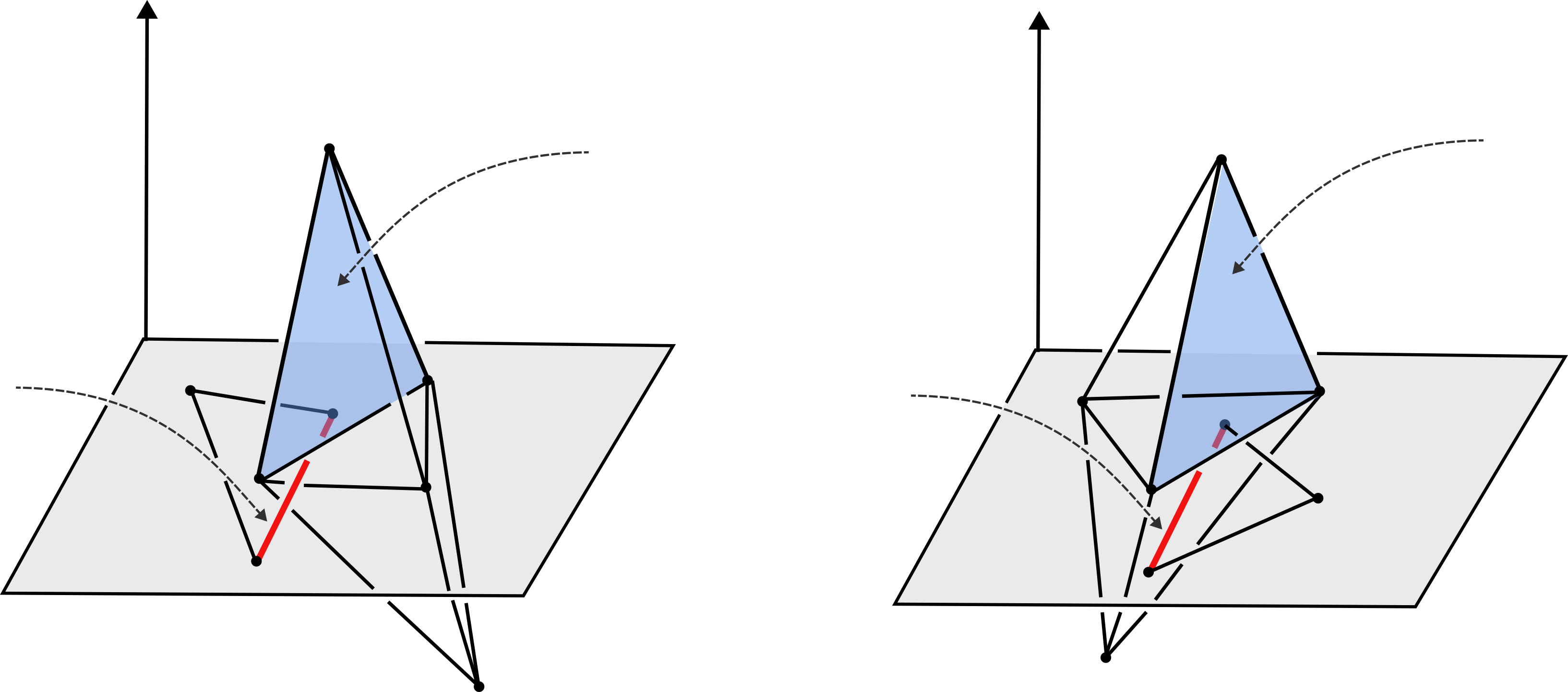}
\put(-294,132){$w$}
\put(-118,130){$w$}
\put(-340,57){$E_i \subset T$}
\put(-165,57){$E_i \subset T'$}
\put(-190,105){$F \subset \mathcal S(\bar T)$}
\put(-15,106){$F \subset \mathcal S(\bar T')$}
    \caption{An example embedding of the suspension $\mathcal S(K_6)$ with the pair of $1$-cycles $T$ and $T'$ containing $E_i$ that are disjoint from the $2$-cycles $\mathcal S(\bar T)$ and $\mathcal S(\bar T')$ containing the face $F$.}
    \label{disjointcycles}
\end{figure}

Since the tally of linking numbers between disjoint $1$ and $2$-cycles are unchanged mod $2$ throughout the homotopy, the value of $\Lambda$ for the embedding of the suspension $\mathcal S(K_6)$ after the homotopy remains non-trivial, as desired. \qed
\\

Our main result, Theorem \ref{mainthm} stated in the introduction, follows directly from Lemma \ref{lem1} and Theorem \ref{mainproof}. 

\subsection{Future directions}

We end by presenting some remaining challenges and questions. 

\begin{enumerate}
    \item Are there other operations (similar to suspension) that map  intrinsically linked graphs in 3 dimensions to intrinsically linked 2-complexes in 4 dimensions?
    \item The operation of exchanging a $3$-cycle in a graph with a ``Y" is known to preserve intrinsic linkedness \cite{ROBERTSON}. Are there analogous operations on a $2$-complex that preserve this property? Operations are known (see \cite{embed2cx} for example) that preserve the embeddability of a $2$-complex in $\mathbb{R}^4$; these moves along with other higher dimensional analogues featured in \cite{nikkuni} are perhaps good places to start.
    \item Conway and Gordon \cite{CandG} also show that any embedding of $K_7$ into $\mathbb{R}^3$ has a cycle with non-trivial Arf invariant -- this cycle is non-trivially knotted. Can we mimic this result for $2$-complexes in 4D using some invariant of knotted $2$-spheres or surfaces? 
\end{enumerate}

\newpage

\bibliographystyle{plain}
\bibliography{references.bib}

\end{document}